\documentclass[12pt, reqno]{amsart}
\usepackage{amsmath, amsthm, amscd, amsfonts, amssymb, graphicx, color, mathrsfs}
\usepackage[bookmarksnumbered, colorlinks, plainpages]{hyperref}
\usepackage[all]{xy} 
\usepackage{slashed}
\usepackage{cite}
\usepackage{xcolor}
\usepackage{listings}
\usepackage{xcolor}
\usepackage{comment}

\renewcommand{\Im}{\mathrm{Im}\,}
\renewcommand{\Re}{\mathrm{Re}\,}

\newtheorem{theorem}{Theorem}[section]
\newtheorem{lemma}[theorem]{Lemma}

\theoremstyle{definition}
\newtheorem{definition}[theorem]{Definition}

\theoremstyle{remark}
\newtheorem{remark}[theorem]{Remark}
\numberwithin{equation}{section}

\begin{document}
\setcounter{page}{1}

\color{darkgray}{
\noindent \centering
{\small   }\hfill    {\small }\\
{\small }\hfill  {\small }}

\centerline{}

\centerline{}


\title[Boundedness of Fourier Integral Operators with complex phases]{Boundedness of Fourier Integral Operators with complex phases on Fourier Lebesgue spaces}

\author[D. Cardona]{Duv\'an Cardona$^{1,*}$}
\address{
  Duv\'an Cardona:
  \endgraf
  Department of Mathematics: Analysis, Logic and Discrete Mathematics
  \endgraf
  Ghent University, Belgium
  \endgraf
  {\it E-mail address} {\rm duvanc306@gmail.com, duvan.cardonasanchez@ugent.be}
  }

\author[W.O. Denteh]{William Obeng-Denteh$^{2,*}$}
\address{
  William Obeng-Denteh
  \endgraf
  Department of Mathematics
  \endgraf
  Kwame Nkrumah University of Science and Technology, (KNUST)- Ghana.
  \endgraf
  {\it E-mail address} {\rm wobengdenteh@gmail.com}
  }

\author[F. Opoku]{Frederick Opoku$^{3,*}$}
\address{
  Frederick Opoku
  \endgraf
  Department of Mathematics
  \endgraf
  Kwame Nkrumah University of Science and Technology, (KNUST)- Ghana.
  \endgraf
  {\it E-mail address} {\rm frederick@aims.edu.gh}
  }

\thanks{Duv\'an Cardona has been  supported  by the FWO  Odysseus  1  grant  G.0H94.18N:  Analysis  and  Partial Differential Equations and by the Methusalem programme of the Ghent University Special Research Fund (BOF)
(Grant number 01M01021), by the FWO Fellowship
Grant No 1204824N and by the FWO Grant K183725N of the Belgian Research Foundation FWO. Frederick Opoku has been supported by Kwame Nkrumah University of Science and Technology, (KNUST) of Ghana.
}

\begin{abstract} In this paper, we establish boundedness estimates for Fourier integral operators on Fourier Lebesgue spaces when the associated canonical relation is parametrised by a complex phase function. Our result constitutes the complex analogue of those obtained for real canonical relations by Rodino, Nicola, and Cordero. We prove that, under the spatial factorization condition of rank $\varkappa$, the corresponding Fourier integral operator is bounded on the Fourier Lebesgue space $\mathcal{F}L^p,$ provided that the order $m$ of the operator satisfies that
$ 
m \leq -\varkappa\left|\frac{1}{p}-\frac{1}{2}\right|, 1 \leq p \leq \infty.
$
This condition on the order $m$ is sharp.
\end{abstract} \maketitle
\allowdisplaybreaks
\tableofcontents


\section{Introduction}

Fourier integral operators with complex-valued phase functions arise, for instance, in the analysis of the wave equation associated with Hörmander sub-Laplacians on Heisenberg-type and Métivier groups, as was recently observed by Martini and Müller \cite{MMNG2023}. Such operators appear in settings where the classical theory of Fourier integral operators with real-valued phase functions is not applicable, in particular when one seeks sharp boundedness results. Indeed, in the study of hyperbolic equations involving the time variable, one encounters topological obstructions (caustics) that cannot be avoided if one restricts oneself to real-valued phase functions. The fact that a single complex-valued phase function can parametrize a canonical relation leads to a global construction; see Laptev, Safarov, and Vassiliev \cite{Laptev:Safarov:Vassiliev}. As a consequence, the theory of Fourier integral operators with complex-valued phase functions does not admit a simple reduction to the real-valued case. Nevertheless, extending the calculus of Fourier integral operators developed by Hörmander \cite{Hormander1971Ac} and by Duistermaat and Hörmander \cite{Duistermaat-Hormander:FIOs-2}, the corresponding calculus for complex phase functions was introduced in the 1970s by Melin and Sjöstrand \cite{Melin-Sjostrand1976}, motivated by the construction of parametrices for operators of principal type with non-real principal symbols.

The aim of this paper is to establish boundedness properties for Fourier integral operators on Fourier Lebesgue spaces associated with complex canonical relations. Our analysis builds on fundamental progress in the subject, particularly in settings where the canonical relation is parametrised by a complex phase function. The methods and techniques developed here combine ideas from Ruzhansky \cite{Ruz2001}, Laptev, Safarov, and Vassiliev \cite{Laptev:Safarov:Vassiliev}, Rodino, Nicola, and Cordero \cite{CNR2009}, Nicola \cite{Nicola2010}, as well as the work of the first author with Ruzhansky \cite{CR2024}.

Fourier integral operators with real-valued phase functions, and associated with real canonical relations, were introduced by H\"ormander \cite{Hormander1971Ac} and Duistermaat and H\"ormander \cite{Duistermaat-Hormander:FIOs-2}. The time-dependent boundedness properties of these operators are well understood in view of the celebrated results of Seeger, Sogge, and Stein \cite{SSS} for $L^p$ spaces, and by the weak (1,1) boundedness theorem due to Tao \cite{Tao}. On the other hand, by taking averages in time,  Sogge \cite{Sogge1991} conjectured \emph{local smoothing estimates} for the wave equation, and consequently for Fourier integral operators. A complete solution to this conjecture in dimension $d=2$ was established by Guth, Wang, and Zhang \cite{GWZ2020}.

On the other hand, motivated by the problem of constructing parametrices or fundamental solutions for operators of principal type with non-real principal symbols, Melin and Sj\"ostrand \cite{Melin-Sjostrand1976} investigated Fourier integral operators associated with complex-valued phase functions. The $L^p$-regularity properties of these operators were further studied by Ruzhansky \cite{Ruz2001}, who showed that the results of Seeger, Sogge, and Stein \cite{SSS} can be extended to this setting. Moreover, in a series of papers, Ruzhansky investigated whether the factorization condition is necessary to obtain the boundedness of these operators; see \cite{Ruz2001} and the references therein.

The regularity properties of Fourier integral operators change significantly when one considers operators associated with degenerate or complex phase functions. This was observed by Rodino, Nicola, and Cordero \cite{CNR2009}, and for phases with a specific order of degeneracy by Nicola \cite{Nicola2010}. In the results of Cordero, Rodino, and Nicola, the operators are associated with real canonical relations. To provide a more complete picture, these results motivate the study of Fourier integral operators whose canonical relation is a complex Lagrangian manifold. In some sense, the use of complex phases provides a more natural approach to Fourier integral operators. Indeed, the complex phase approach allows one to avoid the geometric obstructions present in the global theory with real phase functions, and notably, every Fourier integral operator with a real phase can be globally parametrised by a single complex phase; see Laptev, Safarov, and Vassiliev \cite{Laptev:Safarov:Vassiliev} for this construction.

In order to present our main result we introduce the required notation in Section \ref{Preliminaries}. In particular,  the class of Fourier integral operators of order $m\in \mathbb{R},$ on $\mathbb{R}^d$  is denoted by $I^m(\mathbb{R}^d,\mathbb{R}^d;\mathcal{C}),$ and the Fourier Lebesgue spaces by  $\mathcal{F}L^p.$  Here, $\mathcal{C}$ is a complex-canonical relation locally parametrised by a complex-valued phase function $\Phi(x,\eta)$ which we assume to satisfy the {\it spatial smooth factorization condition} (SSFC) in Definition \ref{SFC}. This assumption
is inspired by, and closely follows, the spatial smooth factorization
condition introduced by Nicola \cite{Nicola2010} in the setting of Fourier integral operators of real-valued phases. It is the spatial counterpart of the one introduced by Ruzhansky in  \cite{Ruz2001} and by Seeger, Sogge and Stein \cite{SSS}. 
The following is our main result, inspired by  the case $\text{Im}(\Phi)\equiv 0$ that have been proved by Nicola \cite{Nicola2010}.
\begin{theorem}\label{Main:theorem}
Let $T\in I^m(\mathbb{R}^d,\mathbb{R}^d;\mathcal{C})$ be a Fourier integral operator associated to a complex canonical relation $\mathcal{C}$ (locally) parametrised by a complex phase function $\Phi$ of positive type.  
Let us assume that there exists a real parameter $\tau\in\mathbb{R},$ and an integer
$0\le \varkappa\le d,$ such that the non-degenerate real-valued phase function
\[
\Phi_\tau(x,\eta):=\Re\Phi(x,\eta)+\tau\,\Im\Phi(x,\eta),
\]
satisfies the spatial factorization condition (SSFC) of rank $\varkappa.$ Then for every $1\le p\le\infty,$ the operator $T:\mathcal{F}L^p(\mathbb{R}^d)\longrightarrow \mathcal{F}L^p(\mathbb{R}^d)$ extends to a bounded operator, that is, 
\begin{align*}
\exists C>0,\,\forall f\in  \mathcal{F}L^p,\,
\|Tf\|_{\mathcal{F}L^p}\lesssim \|f\|_{\mathcal{F}L^p},
\end{align*}provided that $ 
m \leq -\varkappa\left|\frac{1}{p}-\frac{1}{2}\right|.$ This condition on the order $m$ is sharp. 
\end{theorem}
Next, we briefly discuss our main theorem.
\begin{remark}
    The main difficulty in proving Theorem~\ref{Main:theorem} arises from the fact that, after conjugating the operator $T$ with the Fourier transform, it can be realised as an operator with phase
\begin{equation}\label{Phi:tau:eta}
    \Phi_{\tau}(\eta,y)=\operatorname{Re}\Phi(\eta,y)+\tau\,\operatorname{Im}\Phi(\eta,y),
\end{equation}
and symbol $ 
\sigma_{\tau}(\eta,y)=\sigma(\eta,y)e^{-2\pi(1+i\tau)\operatorname{Im}\Phi(\eta,y)}\in S^m_{1/2,1/2}(\mathbb{R}^d_\eta\times \mathbb{R}^d_y).$ 
The main point is that the symbol $\sigma_{\tau}\in S^m_{1/2,1/2}$ now exhibits a more oscillatory behaviour, since it belongs to a H\"ormander class of type $(\tfrac12,\tfrac12)$, rather than to the Kohn--Nirenberg class $S^m_{1,0}$ according to the analysis in \cite{Nicola2010}. Even in the case of pseudo-differential operators, the behavior is markedly different. 
While pseudo-differential operators with Kohn--Nirenberg symbols fall within the Calder\'on--Zygmund theory, 
dealing with H\"ormander classes requires the use of Littlewood--Paley theory; see Fefferman~\cite{Fefferman1973}. Our motivation in building our analysis using the phase function \ref{Phi:tau:eta} comes from the approach in \cite{CR2024} where the weak (1,1) inequality for FIOs with complex phases has been established extending a previous result for real-valued phases due to T. Tao \cite{Tao}.

In the setting of real-valued phase functions, as in \cite{Nicola2010}, one must employ a first decomposition 
of the symbol into Littlewood--Paley components and, in addition, a second angular decomposition, which is typical 
in the analysis introduced by Seeger, Sogge, and Stein~\cite{SSS}. Our main contribution is to show that, even in the 
presence of the oscillatory symbol $\sigma_{\tau}\in S^{m}_{1/2,1/2}$, the estimates arising from the second 
decomposition yield sufficiently accurate kernel bounds to ensure the boundedness of the operator.
\end{remark}
\begin{remark}
   The sharpness of the order condition above as been discussed when  $\text{Im}(\Phi)\equiv 0$ in  Nicola \cite[Page 209]{Nicola2010}, see also Rodino, Nicola, and Cordero \cite{CNR2009}. Our proof of Theorem \ref{Main:theorem} will combine the approach and methods in Nicola \cite{Nicola2010} adapted to the setting of complex phases as developed by Ruzhansky \cite[Page 47]{Ruz2001}.  
\end{remark}
\begin{remark}
     As it was observed by Nicola \cite{Nicola2010}, the spatial factorization condition (SSFC) of rank $\varkappa$ in Definition \ref{SFC} implies that the Hessian
$d_x^2\Phi_\tau(x,\eta)$ has rank at most $\varkappa$. In particular, this condition is automatically satisfied when
$\varkappa=d$, or whenever the Hessian $d_x^2\Phi_\tau(x,\eta)$ has constant rank $\varkappa$.
This includes, as a limiting case, phases that are linear in the spatial
variable $x$, corresponding to $\varkappa=0$ and then to the case of pseudo-differential operators.
\end{remark}
\begin{remark}
    Observe, for instance, that in the case $\varkappa = d$, the order condition
\[
m \leq -d\left|\frac{1}{p}-\frac{1}{2}\right|,
\]
for boundedness on Fourier Lebesgue spaces differs dramatically from the order condition
\[
m \leq -(d-1)\left|\frac{1}{p}-\frac{1}{2}\right|,
\]
established by Seeger, Sogge, and Stein \cite{SSS} for the boundedness of Fourier integral operators with real-valued phase functions on $L^p$-spaces. In particular, the endpoint boundedness results at $p=1,\infty$ generally fail. 

Nevertheless, Ruzhansky \cite{Ruz2001} showed that the Seeger--Sogge--Stein order can also be preserved for Fourier integral operators with complex-valued phase functions.
\end{remark}
\begin{remark}
Our estimates assume the symbol to be compactly supported in the spatial variable and it would be interesting to investigate under wich conditions on the size of the phase our boundedness result remains global. We refer the reader for this topic to the references due to Staubach, Rodr\'iguez Lopez, Israelsson and Mattsson \cite{Staubach2021,Staubach2023,RodriguezLopez2014,RodriguezLopez2015}, see also Ruzhansky and Sugimoto \cite{Ruzhansky2006, Ruzhansky:Sugimoto}. 
\end{remark}

This paper is organised as follows.
In Section \ref{Preliminaries}  we collect the basic definitions, notation, and results on pseudodifferential and Fourier integral operators. The presentation is designed to provide a self-contained and
systematic framework for the analysis carried out in the subsequent sections,
with particular emphasis on symbol classes, phase functions, canonical
relations, and mapping properties of Fourier integral operators on
Fourier--Lebesgue spaces. In Section \ref{Proof:main:theorem:section} we present the derivation of kernel estimates, endpoint boundedness results, and
interpolation arguments for the proof of Theorem \ref{Main:theorem}.

\section{Preliminaries}\label{Preliminaries}

For non–negative quantities $A$ and $B$, the notation $A\lesssim B$ means that
there exists a constant $C>0$, independent of the relevant parameters, such that
$A\le CB$.  
We write $A\asymp B$ if both $A\lesssim B$ and $B\lesssim A$ hold and we denote by $N_0$ a fixed positive integer, independent of the dyadic indices. Now, for $R>0, \,x_0\in \mathbb{R}^d$, we denote by $B_R(x_0)$ the open ball in $\mathbb{R}^d$ with centre $x_0$ and radius $R$. We also set 
\[
\nabla_1 := (\partial_{\eta_1},\,\partial_{\eta_2},\,\ldots,\,\partial_{\eta_d}),\]  to denote the gradient operator with respect to the frequency variable $\eta$ of the phase space variable $(x,\eta).$  For an open set $U\subset\mathbb{R}^{d}$, we denote by $C^\infty(U)$ the space of
smooth functions on $U$ and by $C_c^\infty(U)$ the subspace of compactly supported
functions. Their duals are denoted respectively by $\mathcal{D}'(U)$ and
$\mathcal{E}'(U)$.  
The Schwartz space on $\mathbb{R}^d$ is denoted by $\mathcal{S}(\mathbb{R}^d)$ and
its dual by $\mathcal{S}'(\mathbb{R}^d)$. The Fourier transform of $f\in\mathcal{S}(\mathbb{R}^d)$ is defined by
\[
\widehat{f}(\eta)=\int_{\mathbb{R}^d} e^{-2\pi i x\cdot\eta} f(x)\,dx,
\]
and extended to $\mathcal{S}'(\mathbb{R}^d)$ by duality.
For $1\le p\le\infty,$ and $s\in\mathbb{R}$, the Fourier--Lebesgue space
$\mathcal{F}L^p_s(\mathbb{R}^d)$ consists of all tempered distributions
$f\in\mathcal{S}'(\mathbb{R}^d)$ such that
\[
\|f\|_{\mathcal{F}L^p_s}
=
\|\langle\eta\rangle^s \widehat{f}(\eta)\|_{L^p(\mathbb{R}^d)}<\infty,
\qquad
\langle\eta\rangle=(1+|\eta|^2)^{1/2},
\]
with the usual modification if $p=\infty$. The operator $\langle D\rangle^s$ defines an isomorphism from
$\mathcal{F}L^p_s$ onto $\mathcal{F}L^p$. Consequently, boundedness results on
$\mathcal{F}L^p_s$ can be reduced to boundedness on $\mathcal{F}L^p$ after a
change of order of the symbol. Let $X\subset\mathbb{R}^d$ be open and $m\in\mathbb{R}$.  
For $0\le\rho,\delta\le1$, the symbol class $S^m_{\rho,\delta}(X\times\mathbb{R}^d)$
is the space of all functions $a\in C^\infty(X\times\mathbb{R}^d)$ such that, for
every compact set $K\subset X$ and for all multi–indices $\alpha,\beta$, there
exists a constant $C_{\alpha,\beta,K}$ satisfying
\[
|\partial_x^\alpha\partial_\eta^\beta a(x,\eta)|
\le C_{\alpha,\beta,K}\,
\langle\eta\rangle^{m-\rho|\beta|+\delta|\alpha|},
\qquad (x,\eta)\in K\times\mathbb{R}^d.
\]
Of particular importance in this work are the classes $S^m_{1,0}$ and
$S^m_{1/2,1/2}$, which naturally arise in the study of Fourier integral operators
and their dyadic decompositions. Indeed, given a symbol $a\in S^m_{\rho,\delta}$, the corresponding pseudodifferential
operator is defined by
\[
a(x,D)f(x)
=
\int_{\mathbb{R}^d} e^{2\pi i x\cdot\eta} a(x,\eta)\widehat{f}(\eta)\,d\eta.
\]
A pseudodifferential operator is called \emph{regularizing} if its symbol belongs
to the Schwartz space $\mathcal{S}(\mathbb{R}^{2d})$. Such operators map
$\mathcal{S}'(\mathbb{R}^d)$ continuously into $\mathcal{S}(\mathbb{R}^d)$ and have
smooth integral kernels.

Let $X,Y\subset\mathbb{R}^d$ be open subsets. A real-valued function
$\Phi\in C^\infty(X\times Y\times\mathbb{R}^d\setminus\{0\})$ is called a
\emph{phase function} if it is positively homogeneous of degree one in $\eta$ and
satisfies
\[
\nabla_{(x,y,\eta)}\Phi(x,y,\eta)\neq0
\quad\text{for all }(x,y,\eta).
\]
The associated canonical relation $\mathcal{C}\subset T^*X\times T^*Y$ is defined by
\[
\mathcal{C}
=
\{(x,\nabla_x\Phi(x,y,\eta),\,y,-\nabla_y\Phi(x,y,\eta)) :
\nabla_\eta\Phi(x,y,\eta)=0\}.
\]As we will observe later, the rank properties of the mixed Hessian $\partial_x\partial_\eta\Phi$ play a
central role in the mapping properties of the associated Fourier integral
operator. On the other hand, in order to introduce Fourier integral operators with complex phases, we require  $\Phi$ to be a phase function of positive type; this means that:
    \begin{itemize}
        \item  $\operatorname{Im} \Phi(x,y,\eta)\ge 0$,
        \item $\Phi(x,y,\eta)$ has no critical points on its domain: $\partial_{\theta} \Phi(x,y,\eta)\neq 0.$
        \item $\Phi(x,y,\eta)$ is homogeneous of degree one in $\eta$ i.e, $\forall \lambda>0, \Phi(x,y,\lambda\eta)=\lambda\Phi(x,y,\eta).$
        \item $\{\partial_\eta\Phi(x,y,\eta)=0\}$ is smooth. This means that if $\partial_\eta\Phi(x,y,\eta)=0,$ then the  vectors $d\frac{\partial\Phi}{\partial\eta_j}$ are linearly independent over $\mathbb{C}.$
    \end{itemize}
    The associated canonical relation $\mathcal{C}\subset \widetilde{T^*X\times T^*Y}$ is defined by
\[
\mathcal{C}
=
\{(x,\nabla_x\Phi(x,y,\eta),\,y,-\nabla_y\Phi(x,y,\eta)) :
\nabla_\eta\Phi(x,y,\eta)=0\},
\] where $\widetilde{M}$ denotes the almost analytic continuation of a real smooth manifold $M,$ see \cite{Melin-Sjostrand1976}.
Let $\sigma:=\sigma(x,\eta)\in S^m_{\rho,\delta}$ be a
symbol compactly supported in the $x$ variable and let  $\Phi$ to be a phase function of positive type. The Fourier integral operator $T$ associated with $\Phi$ and $\sigma$ is
 defined by
\begin{equation}\label{FIO}
  Tf(x)
=
\int_{\mathbb{R}^d}\int_{\mathbb{R}^d}
e^{2\pi i \Phi(x,y,\eta)}\sigma(x,\eta)f(y)\,d\eta\,dy.  
\end{equation}
Under standard non–degeneracy assumptions on $\Phi$, the operator $T$ defines a
regular operator mapping $C_c^\infty(Y)$ into $\mathcal{D}'(X)$ and extends by density
to various function spaces. The class of Fourier integral operators with symbols in the class $\sigma\in S^m_{\rho,1-\rho},$ $1/2\leq \rho\leq 1,$ is invariantly defined under changes of coordinates and denoted by $I^m_{\rho,1-\rho}(X,Y;\mathcal{C}),$ indicating that their wave-front sets are contained in the canonical relation $\mathcal{C},$ locally parametrised by the phase function $\Phi.$ When $\rho=1,$ we write  $I^m(X,Y;\mathcal{C}):=I^m_{1,0}(X,Y;\mathcal{C}).$

In view of the H\"ormander phase-function theorem \cite{Hormander1971Ac} in the setting of complex-valued phase functions \cite{Melin-Sjostrand1976}, microlocally, each Fourier integral operator \eqref{FIO} can be re-written as 
\begin{align*}
Tf(x)=\int_{\mathbb{R}^d} e^{2\pi i \Phi(x,\eta)}\,\sigma(x,\eta)\,\widehat{f}(\eta)\,d\eta= \int_{\mathbb{R}^d}\int_{\mathbb{R}^d} e^{2\pi i \Phi(x,\eta)-2\pi i y\cdot \eta}\,\sigma(x,\eta)f(y)dy\,d\eta.
\end{align*} So, from now we will assume that each operator in the class  $I^m_{\rho,1-\rho}(X,Y;\mathcal{C}),$ $1/2\leq \rho\leq 1,$ has microlocally this form. 

In order to obtain sharp Fourier--Lebesgue boundedness results for Fourier integral operators with complex phase, we shall impose a structural assumption on the spatial dependence of the phase function. This assumption
is inspired by, and closely follows, the spatial smooth factorization
condition introduced by Nicola \cite{Nicola2010} in the setting of Fourier integral operators of real-valued phases.
Let us consider $\Phi(x,\eta)=\Re\Phi(x,\eta)+ i\,\Im\Phi(x,\eta)$ to be a
complex-valued phase function of positive type. For a fixed real parameter
$\tau\in\mathbb{R}$, we associate to $\Phi$ the real-valued phase
\[
\Phi_\tau(x,\eta):=\Re\Phi(x,\eta)+\tau\,\Im\Phi(x,\eta).
\]
The set $\Lambda_\tau\subset\mathbb{R}^d\times(\mathbb{R}^d\setminus\{0\})$
denotes the open domain on which $\Phi_\tau$ is smooth. Typically, we will assume that $\Lambda_\tau$ is the support of the symbol of the corresponding Fourier integral operator.
For $(x,\eta)\in\Lambda_\tau$, we write $\nabla_x\Phi_\tau(x,\eta)$ for the
gradient of $\Phi_\tau$ with respect to the spatial variable $x$, and
\[
d_x^2\Phi_\tau(x,\eta)
:=\bigl(\partial_{x_j}\partial_{x_k}\Phi_\tau(x,\eta)\bigr)_{1\le j,k\le d}
\]
for the corresponding spatial Hessian matrix.

\begin{definition}[Spatial smooth factorization condition (SSFC)]\label{SFC}
We assume that there exists an integer $0\le \varkappa\le d$ such that, for some 
$\tau \in \mathbb{R},$ and for every $(x_0,\eta_0)\in\Lambda_\tau$ with
$\eta_0\in\mathbb{S}^{d-1}$, there exists an open neighbourhood $\Omega$ of $x_0$
and an open neighbourhood $\Gamma'\subset\mathbb{S}^{d-1}$ of $\eta_0$, with
$\Omega\times\Gamma'\subset\Lambda_\tau$, satisfying the following properties.
\begin{itemize}
    \item {{(SSFC)}} For every $\eta\in\Gamma'$ there exists a smooth fibration of $\Omega$,
depending smoothly on $\eta$, whose fibres are affine subspaces of codimension
$\varkappa$, such that
\[
\nabla_x\Phi_\tau(\cdot,\eta)
\quad\text{is constant along each fibre}.
\]
\item In addition, we also  assume that $\Phi_\tau$ holds the non-degeneracy condition,
$$\operatorname{det}\partial_y\partial_{\eta}(\operatorname{Re}\Phi(y,\eta)+\tau\operatorname{Im}\Phi(y,\eta))\neq 0.$$ 
\end{itemize}
To simplify the terminology, if a phase function $\Phi$ can be associated with a phase function $\Phi_\tau$ that satisfies both properties in Definition \ref{SFC} for some $\tau \in \mathbb{R}$, we say that $\Phi$ satisfies the smooth spatial factorization condition of rank $\varkappa$.
\end{definition}
\begin{remark}
    As in the case of real-valued phase functions, this condition implies that the Hessian
$d_x^2\Phi_\tau(x,\eta)$ has rank at most $\varkappa$. In particular, the condition is automatically satisfied when
$\varkappa=d$, or whenever the Hessian $d_x^2\Phi_\tau(x,\eta)$ has constant rank $\varkappa$.
This includes, as a limiting case, phases that are linear in the spatial
variable $x$, corresponding to $\varkappa=0$.
\end{remark}
\begin{remark}
    Observe that in Definition \ref{SFC}, if $\Phi_\tau$ holds the non-degeneracy condition for some $\tau\in \mathbb{R},$ since $\gamma\mapsto\operatorname{det}\partial_y\partial_{\eta}(\operatorname{Re}\Phi(y,\eta)+\gamma\operatorname{Im}\Phi(y,\eta)),$ is a polynomial function, $\Phi_\tau$ holds the non-degeneracy condition  for almost all $\tau\in \mathbb{C}.$
\end{remark}
Under this assumption, the geometry of the canonical relation associated with
$\Phi_\tau$ exhibits a partial flatness in the spatial variables, quantified
by the integer $\varkappa$. As shown by Nicola \cite{Nicola2010} in the real-phase setting, this geometric
feature governs the precise loss of derivatives in Fourier--Lebesgue spaces.
In particular, it leads to the sharp threshold
\[
m \le -\, \varkappa\,\Bigl|\tfrac12-\tfrac1p\Bigr|,
\]
which will reappear in the complex-phase estimates proved below.

Throughout the paper we use standard dyadic partitions of unity in the frequency
variable. Let $\{\psi_j\}_{j\ge0}$ be a smooth partition of unity such that
$\operatorname{supp}\psi_j\subset\{\eta:2^{j-1}\le|\eta|\le2^{j+1}\}$ for $j\ge1$,
and $\operatorname{supp}\psi_0\subset\{|\eta|\le2\}$. This decomposition allows us to localize Fourier integral operators at frequency
scale $|\eta|\sim2^j$ and to exploit both symbol estimates and angular
decompositions adapted to the geometry of the canonical relation. This decomposition will be repeatedly used in the proof of boundedness results
on Fourier--Lebesgue spaces, especially in the treatment of endpoint cases and
interpolation arguments.

\section{FIOs $\&$ affine fibrations}\label{Proof:main:theorem:section}

The aim of this section is to prove Theorem \ref{Main:theorem}. 
The Fourier integral operator (FIO) $T$ to be considered is defined as
\begin{align}
Tf(x)= \int_{\mathbb{R}^d} e^{2\pi i\, \Phi(x,\eta)}\, \sigma(x,\eta)\, \widehat{f}(\eta)\, d\eta,
\end{align}
where $T:\mathcal{S}(\mathbb{R}^d)\to C^{\infty}_0(\mathbb{R}^{d})$, since the symbol $\sigma:=\sigma(x,\eta)\in S^m_{1,0}
(\mathbb{R}^d\times \mathbb{R}^d)$ has compact support in $x$.  Conjugating the operator $T$ with the Fourier transform we have that,
     \begin{align}
     \tilde{T}f(x)=\displaystyle{\iint} e^{2\pi i(\Phi(\eta,y)-x\eta)}\sigma(\eta,y)f(y)\,dyd\eta, \label{71}
     \end{align}
     where $\tilde{T}=\mathcal{F}\circ T \circ \mathcal{F}^{-1}$. The symbol $\sigma(\eta,y)$ is no longer supported with respect to $y$ and the phase $\Phi(\eta,y)-x\eta$  is a real-valued function which is no longer positively homogeneous of order $1$ with respect to $\eta$.

Let us assume that there exists a real parameter $\tau\in\mathbb{R},$ and an integer
$0\le \varkappa\le d,$ such that the real-valued phase
\[
\Phi_\tau(x,\eta):=\Re\Phi(x,\eta)+\tau\,\Im\Phi(x,\eta)
\]
satisfies the spatial factorization condition (SSFC) of rank $\varkappa$ in Definition \ref{SFC}.
    With this choice of $\tau$, we rewrite (\ref{71})  as follows:
     \begin{align}
     \tilde{T}f(x)=&\displaystyle{\iint} e^{2\pi i (\Phi(\eta,y)-x\eta)}\sigma(\eta,y)f(y)\,dyd\eta.\nonumber\\
     =&\displaystyle{\iint} e^{2\pi i(\operatorname{Re}\Phi(\eta,y)+i\operatorname{Im}\Phi(\eta,y))}\sigma(\eta,y)e^{-2\pi ix\eta}f(y)\,dyd\eta.\nonumber\\
     =&\displaystyle{\iint} e^{2\pi i(\operatorname{Re}\Phi(\eta,y)-\tau\operatorname{Im}\Phi(\eta,y)+\tau\operatorname{Im}\Phi(\eta,y)+i\operatorname{Im}\Phi(\eta,y))}\sigma(\eta,y)e^{-2\pi ix\eta}f(y)\,dyd\eta.\nonumber\\
     =&\displaystyle{\iint} e^{2\pi i(\operatorname{Re}\Phi(\eta,y)+\tau\operatorname{Im}\Phi(\eta,y))}\sigma(\eta,y)e^{-2\pi(1+i\tau)\operatorname{Im\Phi(\eta,y)}}e^{-2\pi ix\eta}f(y)\,dyd\eta.\nonumber\\
     =&\displaystyle{\iint} e^{2\pi i(\operatorname{Re}\Phi(\eta,y)+\tau\operatorname{Im}\Phi(\eta,y)-x\eta)}\sigma(\eta,y)e^{-2\pi(1+i\tau)\operatorname{Im\Phi(\eta,y)}}f(y)\,dyd\eta.\nonumber\\
     =&\displaystyle{\iint} e^{2\pi i(\Phi_{\tau}(\eta,y)-x\eta)}\sigma_{\tau}(\eta,y)f(y)\,dyd\eta,
     \label{73}
     \end{align}
     where  $\Phi_{\tau}(\eta,y)=\operatorname{Re}\Phi(\eta,y)+\tau\operatorname{Im}\Phi(\eta,y)$. Note that $$\sigma_{\tau}(\eta,y)=\sigma(\eta,y)e^{-2\pi(1+i\tau)\operatorname{Im\Phi(\eta,y)}}\in S^m_{1/2,1/2}(\mathbb{R}^d_\eta\times \mathbb{R}^d_y),$$ since  $e^{-2\pi(1+i\tau)\operatorname{Im\tilde{\Phi}(\eta,y)}}\in S^0_{1/2,1/2},$ see \cite[Page 28]{Ruz2001}.  Then it satisfies the following estimates; \\
     $\big|\partial_y^\alpha\partial_{\eta}^\beta e^{-2\pi(1+i\tau)\operatorname{Im\Phi(\eta,y)}}\big|\lesssim \langle y \rangle^{\dfrac{|\beta|-|\alpha|}{2}}$,  and 
     $\big|\partial_y^\alpha\partial_{\eta}^\beta \sigma_\tau(\eta,y) \big|\lesssim \langle y \rangle^{m-\frac{1}{2}|\alpha|+\frac{1}{2}|\beta|}$.\\
     The operator in (\ref{73}) has integral kernel
     \begin{equation}
         K(x,y)= \displaystyle{\int} e^{2\pi i(\Phi_{\tau}(\eta,y)-x\eta)}\sigma_{\tau}(\eta,y)d\eta.
     \end{equation}
    Now, we consider the Littlewood-Paley decomposition in $y$. We choose a smooth function $\psi_0(y)$ such that $\psi_0(y)=1$ for $|y|\le 1$ with $\psi_0(y)=0$ for $|y|\ge 2$. Set $\psi(y)=\psi_0(y)-\psi(2y)$, $\psi_j(y)=\psi(2^{-j}y), j\ge 1.$ Then,
    \begin{align*}
        \sum_{j=0}^{\infty}\psi_j(y)=1, \forall y \in \mathbb{R}^d\setminus(0).
    \end{align*}
 Notice that if $j\ge 1$, then $\psi_j$ is supported where $2^{j-1}\le |y| \le 2^{j+1} $.
    We can now write the kernel above as,
    \begin{align*}
        K(x,y)=\sum_{j\ge 1}K_j(x,y), 
    \end{align*}
    where,
       \begin{align} K_j(x,y)=\displaystyle{\int} e^{2\pi i(\Phi_{\tau }(y,\eta)-x\eta)}\sigma_{\tau j}(\eta,y)d\eta,\end{align}
    with $\sigma_{\tau j}(y,\eta)=\sigma_{\tau}(y,\eta)\psi_j(y).$ 

    In what follows we will use the argument of Nicola \cite{Nicola2010} adapted to complex phases as in Ruzhansky \cite[Page 47]{Ruz2001}. 
    Observe that $\eta$ lies in the open neighbourhood $\Omega' \Subset \Omega$. After shrinking $\Omega'$ and the parameter set $\Gamma'$ if necessary, the spatial smooth factorization condition (SSFC) (see Definition \ref{SFC}) implies that there exists an integer $\varkappa$ with $0 \le \varkappa\le d$, an open neighbourhood $U \times V$ of $(0,0)$ in $\mathbb{R}^\varkappa \times \mathbb{R}^{d-\varkappa}$, and a smooth change of variables
$U \times V \ni (u,v) \longmapsto \eta_y(u,v) \in \Omega_y, 
$ smoothly depending on the parameter $y \in \Gamma'$ and homogeneous of degree $0$ with respect to $y$, with $\Omega' \subset \Omega_y \subset \Omega$, such that the function $\eta \mapsto \nabla_\eta \Phi_\tau(\eta,y)$ is constant along each $(d-\varkappa)$--dimensional affine plane given by $u=\mathrm{const}$. 

Equivalently, in these adapted coordinates the gradient $\nabla_\eta \Phi_\tau(\eta_y(u,v),y)$ depends only on the transversal variable $u$ and is independent of the fibre variable $v$, so that
$\nabla_\eta \Phi_\tau(\eta_y(u,v),y)=\nabla_\eta \Phi_\tau(\eta_y(u,0),y),\, \,(u,v)\in U\times V$. We emphasize that, due to the previous conjugation of the Fourier integral operator, we are using Definition \ref{SFC}  with the variables $(x,\eta)$ replaced by $(\eta,y)$, so that no additional assumptions are introduced and the factorization structure is preserved.

For every $j\ge 1$ we chose $u^v_j, v=1,\cdots,N_\varkappa(j),$ such that $|u_j^v-u_j^{v^{\prime}}|\ge C_02^{-j/2}$ for $v\neq v^{\prime}$ and such that $U$ is covered by balls with centre $u_j^v$ and radius $C_12^{-j/2}$. It is easy to see that $N_\varkappa(j)=O(2^{j\varkappa/2})$. Let then $\eta_j^v=\eta_y(u_j^v,0)$. Consider a smooth partition of unity to cover $U$, given by smooth functions $\chi_j^v(u), v=1,\cdots,N_\varkappa(j),$ supported in the above balls and satisfying the estimate (see \cite{Nicola2010}),
\begin{align*}
    |\partial_u^{\alpha}\chi_j^v(u)|\lesssim 2^{j|\alpha|/2}.
\end{align*} Note that the support of $\chi_j(u)$ is contained on $U.$ So, in the support of $\chi_j,$ we have that $|u|\lesssim 1.$
    We further decompose the kernel $K_j(x,y)$ as $\sum_{v=1}^{N_\varkappa(j)}K_j^v(x,y)$ with 
    \begin{align}
        K_j^v(x,y)=  \displaystyle{\int}e^{2\pi i (\Phi_{\tau}(y,\eta)-x\eta)}\chi_j^v(u(y,\eta))\sigma_{\tau}^v(\eta,y)\psi_j(y)\,d\eta, \label{76}
    \end{align}
 where, $\sigma_{\tau}^v(y,\eta)\psi_j(y)$ is the symbol    restricted to a small cap  $v$  and localized to the dyadic band  $j$ in  $y$.
Consider now the first order Taylor expansion of $\Phi_{\tau}(\cdot,y)$ at $\eta^v_j$:
\begin{align*}
    \Phi_{\tau}(\eta,y)=\Phi_{\tau}(\eta_j^v,y)+\langle \nabla_1 \Phi_{\tau}(\eta_j^v,y) ,\eta-\eta_j^v\rangle+R_j^v(\eta,y),
\end{align*}
where, \begin{align*}
\nabla_1 := (\partial_{\eta_1},\,\partial_{\eta_2},\,\ldots,\,\partial_{\eta_d}),
\end{align*}
and
\begin{align*}
    R_j^v(\eta,y)=\dfrac{1}{2}\int\limits_0^1(1-t)(d_1^2\Phi_{\tau})(\eta^v_j+t(\eta-\eta_j^v),y)[\eta-\eta_j^v,\eta-\eta_j^v]\,dt.
\end{align*}
The phase function can be expressed as 
\begin{align*}
    \Phi_{\tau}(\eta,y)-\langle x,\eta \rangle= \langle \nabla_1\Phi_{\tau}(\eta_j^v,y)-x,\eta\rangle + r(\eta),
\end{align*}
where,
\begin{align*}
    r(\eta)= \Phi_{\tau}(\eta_j^v,y)-\langle \nabla_1\Phi_{\tau}(\eta_j^v,y),\eta_j^v\rangle+R_j^v(\eta,y).
\end{align*}
Fix $y \in \mathbb{R}^d$, and let us consider the first-order Taylor expansion of $\Phi_\tau(\cdot,y)$ at a point $\eta_j^v$:
\begin{align*}
    \Phi_\tau(\eta,y) 
    = \Phi_\tau(\eta_j^v, y) 
    + \langle \nabla_1 \Phi_\tau(\eta_j^v, y), \eta - \eta_j^v \rangle 
    + R_j^v(\eta,y),
\end{align*}
where $\nabla_1 \Phi_\tau(\eta_j^v,y) = \partial_\eta \Phi_\tau(\eta_j^v,y)$, and the remainder $R_j^v(\eta,y)$ is given by the integral form of the second-order Taylor remainder:
\begin{align*}
    R_j^v(\eta,y) = \frac12 \int_0^1 (1-t)\, d_1^2 \Phi_\tau(\eta_j^v + t(\eta-\eta_j^v), y)[\eta-\eta_j^v, \eta-\eta_j^v]\, dt.
\end{align*}
Here, $d_1^2 \Phi_\tau$ denotes the Hessian matrix of second derivatives with respect to the first argument of $\eta$. Thus, \[
d_1^2 \Phi_{\tau}(\eta,y)
= \left( \frac{\partial^2 \Phi_{\tau}}{\partial \eta_i\, \partial \eta_j}(\eta,y) \right)_{1 \le i,j \le d}.
\]
The kernel of the Fourier integral operator involves the phase $\Phi_\tau(\eta,y) - x\cdot \eta$. Substituting the Taylor expansion:
\begin{align*}
    \Phi_\tau(\eta,y) - x \cdot \eta 
    &= \Phi_\tau(\eta_j^v, y) + \langle \nabla_1 \Phi_\tau(\eta_j^v,y), \eta - \eta_j^v \rangle + R_j^v(\eta,y) - x \cdot \eta \\
    &= \langle \nabla_1 \Phi_\tau(\eta_j^v,y) - x, \eta \rangle 
    + \Phi_\tau(\eta_j^v,y) - \langle \nabla_1 \Phi_\tau(\eta_j^v,y), \eta_j^v \rangle + R_j^v(\eta,y).
\end{align*}
We denote
\begin{align*}
r(\eta) := \Phi_\tau(\eta_j^v,y) - \langle \nabla_1 \Phi_\tau(\eta_j^v,y), \eta_j^v \rangle + R_j^v(\eta,y),
\end{align*}
so that the phase $\Phi_\tau(\eta,y)$ can be written as
\begin{align*}
\Phi_\tau(\eta,y) - x\cdot \eta = \langle \nabla_1 \Phi_\tau(\eta_j^v,y) - x, \eta \rangle + r(\eta).
\end{align*}
This decomposition allows us to apply repeated integration by parts with respect to $\eta$, using operators adapted to the decomposition into $\eta'$, $\eta''$ coordinates, in order to control the kernel of the Fourier integral operator. The linear term provides decay, while the remainder $r(\eta)$ is uniformly bounded and does not affect the estimates.

After linearizing the phase function, we can now rewrite $(\ref{76})$ as,
    \begin{align}        K_j^v(x,y)=\int\limits_{ \text{supp}_\eta(\sigma(\eta,y))} e^{2\pi i \langle \nabla_1\Phi_{\tau}(\eta_j^v,y)-x,\eta\rangle}\sigma_{\tau}^v(\eta,y)\psi_j(y)\,d\eta, \label{77}
    \end{align}
where,
\begin{align*}
    \sigma_{\tau j}^v(\eta,y)=e^{2\pi i r(\eta)}\chi_j^v(u(y,\eta))\sigma_{ \tau j}(\eta,y).
\end{align*}
Note that in the support of $\chi_j^v(u(y,\eta)),$ we have that 
$$|\eta|\lesssim |(y,\eta)|\asymp  |u(y,\eta)|\lesssim 2^{-\frac{j}{2}}.$$
The following lemma will be used later in our argument of integration by parts.

\begin{lemma}
    Let $\Phi_\tau(\eta,y)$ be a smooth phase function and, for fixed $j$ and $v$, consider the splitting $\eta = (\eta^{\prime},\eta^{\prime\prime})$, where $\eta^{\prime\prime}$ are tangent to a small cap $v$. Define the operator
\begin{align*}
L^v_j=(1-\langle2^{-j/2}\nabla_{{\eta}^{\prime}},2^{-j/2}\nabla_{{\eta}^{\prime}}\rangle)(1-\langle\nabla_{{\eta}^{\prime\prime}},\nabla_{{\eta}^{\prime\prime}}\rangle).
\end{align*}
Then $L_j^v$ satisfies the identity
\begin{align*}
     (1+4\pi^22^{-j}|( \nabla_1\Phi_{\tau}(\eta_j^v,y)-x)^{\prime}|^2)(1+4\pi^2|( \nabla_1\Phi_{\tau}(\eta_j^v,y)-x)^{\prime\prime}|^2)e^{2\pi i \langle \nabla_1\Phi_{\tau}(\eta_j^v,y)-x,\eta\rangle}\\
     =L^v_je^{2\pi i \langle \nabla_1\Phi_{\tau}(\eta_j^v,y)-x,\eta\rangle}.
\end{align*}
\end{lemma}
\begin{proof}
Let
\[
A:=\nabla_1\Phi_\tau(\eta_j^v,y)-x
\qquad\text{and}\qquad
e(\eta):=e^{2\pi i\langle A,\eta\rangle}
= e^{2\pi i\langle A',\eta'\rangle}\,e^{2\pi i\langle A'',\eta''\rangle},
\]
where $A=(A',A'')$ corresponds to the splitting of $\eta$. We compute the action of each factor in $L_j^v$ on the exponential.

\medskip\noindent\textbf{(i) Action on the $\eta'$--factor.}
For any coordinate $k$ in the $\eta'$ variables,
\[
\partial_{\eta'_k} e(\eta) = 2\pi i A_k' \, e(\eta),
\]
hence
\[
\partial_{\eta'_k}^2 e(\eta) = (2\pi i A_k')^2 e(\eta) = -4\pi^2 (A_k')^2 e(\eta).
\]
Therefore (using linearity and summing over $k$),
\[
\langle \nabla_{\eta'},\nabla_{\eta'}\rangle e(\eta)
= \sum_{k} \partial_{\eta'_k}^2 e(\eta)
= -4\pi^2 |A'|^2 e(\eta).
\]
Multiplying by the $2^{-j}$ factor from $L_j^v$ we obtain
\[
\langle 2^{-j/2}\nabla_{\eta'},2^{-j/2}\nabla_{\eta'}\rangle e(\eta)
= 2^{-j} \langle \nabla_{\eta'},\nabla_{\eta'}\rangle e(\eta)
= -4\pi^2 2^{-j} |A'|^2 e(\eta).
\]
Thus
\[
\big(1 - \langle 2^{-j/2}\nabla_{\eta'},2^{-j/2}\nabla_{\eta'}\rangle\big) e(\eta)
= \big(1 + 4\pi^2 2^{-j} |A'|^2\big) e(\eta).
\]

\medskip\noindent\textbf{(ii) Action on the $\eta''$--factor.}
Exactly the same computation (without the $2^{-j}$ weight) for the $\eta''$--variables gives
\[
\langle \nabla_{\eta''},\nabla_{\eta''}\rangle e(\eta)
= -4\pi^2 |A''|^2 e(\eta),
\]
and hence
\[
\big(1 - \langle \nabla_{\eta''},\nabla_{\eta''}\rangle\big) e(\eta)
= \big(1 + 4\pi^2 |A''|^2\big) e(\eta).
\]

\medskip\noindent\textbf{(iii) Combine the two factors.}
Since the $\eta'$ and $\eta''$ variables are independent and the two factors
in $L_j^v$ commute when acting on a product of functions that separate in
$\eta'$ and $\eta''$ (in particular on $e(\eta)$), we may apply them successively:
\[
L_j^v e(\eta)
= \big(1 - \langle 2^{-j/2}\nabla_{\eta'},2^{-j/2}\nabla_{\eta'}\rangle\big)
  \big(1 - \langle \nabla_{\eta''},\nabla_{\eta''}\rangle\big)
  e(\eta).
\]
By the two identities above we obtain,
\[ L_j^v e(\eta)=
\big(1 + 4\pi^2 2^{-j} |A'|^2\big)\big(1 + 4\pi^2 |A''|^2\big) e(\eta),
\]
which is exactly the claimed identity after recalling the definition of $A$.
\end{proof}
Now we apply repeated integration by parts to $(\ref{77})$ gives us

\begin{align}
        K_j^v(x,y)=H_j^v(x,y)\int (L_j^v)^N\sigma_{\tau j}^v(\eta,y) e^{2\pi i \langle \nabla_1\Phi_{\tau}(\eta_j^v,y)-x,\eta\rangle}\,d\eta, 
    \end{align}
where 
\begin{align}
    H_j^v(x,y)= (1+4\pi^22^{-j}|( \nabla_1\Phi_{\tau}(\eta_j^v,y)-x)^{\prime}|^2)^{-N}(1+4\pi^2|( \nabla_1\Phi_{\tau}(\eta_j^v,y)-x)^{\prime\prime}|^2)^{-N}.\label{H}
\end{align}
The following estimate can be verified easily and for completness we present a proof.
\begin{lemma}[Equivalence of polynomial decay]\label{EOPD}
Let $N\in\mathbb{N}$. There exists a constant $C_N>0$ such that for all
$z\in\mathbb{R}^d$,
\[
(1+|z|^2)^{-N} \;\le\; C_N\,(1+|z|)^{-2N}.
\]
Consequently, for any $\lambda>0$,
\[
(1+\lambda^2|z|^2)^{-N}
\;\le\;
C_N\,(1+\lambda|z|)^{-2N}.
\]
\end{lemma}

\begin{proof}
For every $z\in\mathbb{R}^d$ we have
\[
1+|z|^2
\;\ge\;
\frac12\bigl(1+|z|\bigr)^2.
\]
Indeed, expanding the right-hand side gives
\[
\frac12(1+2|z|+|z|^2)\le 1+|z|^2.
\]
Raising both sides to the power $-N$ yields
\[
(1+|z|^2)^{-N}
\;\le\;
2^N(1+|z|)^{-2N}.
\]
Setting $C_N=2^N$ proves the first claim.

The second inequality follows by applying the first with $z$ replaced by
$\lambda z$, which gives
\[
(1+\lambda^2|z|^2)^{-N}
\le
2^N(1+\lambda|z|)^{-2N}.
\]
This completes the proof.
\end{proof}
From Lemma~(\ref{EOPD}), we can express $H^v_j(x,y)$ as 
\begin{align*}
    H_j^v(x,y)\lesssim (1+2^{-j/2}|( \nabla_1\Phi_{\tau}(\eta_j^v,y)-x)^{\prime}|)^{-2N}(1+|( \nabla_1\Phi_{\tau}(\eta_j^v,y)-x)^{\prime\prime}|)^{-2N}.
\end{align*} 
Now, we have that
\begin{align*}
        \big|K_j^v(x,y)\big|&= \big|H_j^v(x,y)\int (L_j^v)^N\sigma_{\tau j}^v(\eta,y) e^{2\pi i \langle \nabla_1\Phi_{\tau}(\eta_j^v,y)-x,\eta\rangle}\,d\eta \big|\\& \lesssim H_j^v(x,y)\int |(L_j^v)^N\sigma_{\tau j}^v(\eta,y)|\,d\eta.
    \end{align*}
We need to find the bound of \\
$|(L_j^v)^N\sigma_{\tau j}^v(\eta,y)|=|(L_j^v)^N \big(e^{2\pi i(\Phi_{\tau}(\eta_j^v,y)-\langle \nabla_1\Phi_{\tau}(\eta_j^v,y),\eta_j^v\rangle+R_j^v(\eta,y))}\chi_j^v(u(y,\eta))\sigma_{ \tau j}(\eta,y)\big)|$.

So we now investigate the effect of  $(L_j^v)^N$ falling on each of the following terms:
\begin{itemize}
    \item  $\chi_j^v(u(y,\eta))$
    \item $\sigma_{ \tau j}(\eta,y)$
    \item $e^{2\pi i(\Phi_{\tau}(\eta_j^v,y)-\langle \nabla_1\Phi_{\tau}(\eta_j^v,y),\eta_j^v\rangle+R_j^v(x,y))}$
\end{itemize}

\begin{lemma}\label{lem1}
Assume that the damped amplitude $\sigma_{\tau}(\eta,y)\in S^m_{\frac{1}{2},\frac{1}{2}}((\mathbb{R}^d\setminus \{0\})\times \mathbb{R}^d) $, 
that is, for every pair of multiindices $\alpha,\beta$ there exists $C_{\alpha,\beta}>0$ such that
\begin{equation}\label{eq2}
\big|\partial_{y}^{\alpha}\partial_{\eta}^{\beta}\sigma_{\tau}(\eta,y)\big|
\le C_{\alpha,\beta}\,\langle y\rangle^{\,m-\frac{|\alpha|}{2}+\frac{|\beta|}{2}},
\qquad (\eta,y)\in(\mathbb{R}^d\setminus \{0\})\times \mathbb{R}^d.
\end{equation}
\end{lemma}
Let $\sigma_{\tau j}(\eta,y)$ denote the localization of $\sigma_{\tau}(\eta,y)$ to the $j$-th dyadic 
$y$–shell, so that on the support of $\sigma_{\tau j}(\eta,y)$ we have $\langle y\rangle\approx 2^{j}$. 
Then, for every multiindex $\beta$ (and uniformly in $y$),
\begin{align}\label{eq3}
\big|\partial_{\eta}^{\beta}\sigma_{\tau j}(\eta,y)\big|
&\le C_{\beta}\,\langle y \rangle^{\,m+\frac{|\beta|}{2}}
\qquad\text{(by \eqref{eq2} with }\alpha=0\text{)}\\[4pt]
&\lesssim C_{\beta}\,2^{jm}\,2^{\,\frac{j|\beta|}{2}}.\nonumber
\end{align}
In particular, each single $\eta$–derivative produces a factor $2^{j/2}$ on the $j$-th shell:
\[
\partial_{\eta_k}\sigma_{\tau j}(\eta,y)=O\big(2^{jm}2^{j/2}\big),
\qquad \langle y \rangle\approx 2^{j}.
\]

\begin{lemma}
Let $\chi_{j}^{v}(u)$ be a smooth cutoff function supported in the region 
\begin{align*}
|u| \lesssim 2^{-j/2},
\end{align*}
and suppose its derivatives satisfy that,
\begin{align*}
|\partial_u^{\alpha}\chi_{j}^{v}(u)| \lesssim 2^{j|\alpha|/2}
\qquad\text{for all multi-indices }\alpha.
\end{align*}
Let $u = u(y,\eta)$ be a smooth function of $(y,\eta)$ such that its derivatives in $\eta$ satisfy the bounds
\begin{align*}
|\partial_{\eta}^{\beta} u(y,\eta)| \lesssim 2^{-j|\beta|/2}
\qquad\text{for all multi-indices }\beta.
\end{align*}
Then for every multi-index $\alpha$ we have
\begin{align*}
\bigl|\partial_{\eta}^{\alpha}\bigl(\chi_{j}^{v}\!\left(u(y,\eta)\right)\bigr)\bigr| \lesssim 1.
\end{align*}
That is, the derivatives of the cutoff $\chi_{j}^{v}$ introduce no net growth: 
the decay factor $2^{-j/2}$ arising from derivatives of $u$ exactly compensates the growth 
$2^{j|\alpha|/2}$ coming from derivatives of $\chi_{j}^{v}$.
\end{lemma}

\begin{lemma}
Let
\[
R_{j}^{v}(\eta,y)
=
\frac12 \int_{0}^{1} (1-t)\,
d_{\eta}^{2}\Phi_{\tau}\!\left(\eta_{j}^{v} + t(\eta-\eta_{j}^{v}),\,y\right)
\bigl[\eta-\eta_{j}^{v},\,\eta-\eta_{j}^{v}\bigr]\; dt
\]
be the remainder term in the first-order Taylor expansion of
$\Phi_{\tau}(\eta,y)$ at $\eta_{j}^{v}$.  
Assume that 
\[
\eta-\eta_{j}^{v}=O(2^{-j/2}).
\]
Then, for every multi-index $\beta$,
\[
\bigl|\partial_{\eta}^{\beta} R_{j}^{v}(\eta,y)\bigr| \lesssim C_{\beta},
\]
where $C_{\beta}$ is independent of $j$. 
\end{lemma}
\begin{remark}
    Since $R_{j}^{v}$ is quadratic in $(\eta - \eta_{j}^{v})$, each derivative 
reduces the degree by at most one.  Given 
\begin{align*}
\eta - \eta_{j}^{v} = O(2^{-j/2}),
\end{align*}
all terms produced by derivatives remain uniformly bounded independently of $j$.
\end{remark}

\begin{lemma}
Let
\[
\sigma_{\tau j}^v(\eta,y)
= \chi_j^v(u(y,\eta))\,\sigma_{\tau j}(\eta,y)\,e^{2\pi i r(\eta)},
\]
where
\begin{enumerate}
  \item $\sigma_{\tau j}(\eta,y)\in S^{m}_{1/2,1/2}$ and 
        $\langle\eta\rangle\sim 2^j$ on its support;
  \item $\chi_j^v(u(y,\eta))$ is supported in an angular cap of radius $O(2^{-j/2})$
        and satisfies $|\partial_u^\alpha \chi_j^v|\lesssim 2^{j|\alpha|/2}$;
  \item the Taylor remainder $R_j^v(\eta,y)$ satisfies 
        \[
        |\partial_\eta^\gamma R_j^v|\lesssim 
        \begin{cases}
        2^{-j/2}, & |\gamma|=1,\\[4pt]
        1, & |\gamma|\ge2,
        \end{cases}
        \quad\text{and}\quad
        R_j^v(\eta,y)=O(2^{-j}) 
        \text{ on the cap}.
        \]
\end{enumerate}
Define
\[
L_j^v=
\Big(1-\langle 2^{-j/2}\nabla_{\eta'},\,2^{-j/2}\nabla_{\eta'}\rangle\Big)
\Big(1-\langle\nabla_{\eta''},\,\nabla_{\eta''}\rangle\Big).
\]
Then for every integer $N\ge1$ there exists a constant $C_N$ such that

\[
\big|(L_j^v)^N \sigma_{\tau j}^v(\eta,y)\big|
\le C_N\,2^{jm}.
\] Note that this estimate coincides with (3.7) in Nicola \cite{Nicola2010}.
\end{lemma}
\begin{proof}
 Write
\[
A_j := 1 - \langle 2^{-j/2}\nabla_{\eta'},\,2^{-j/2}\nabla_{\eta'}\rangle,
\qquad
B := 1 - \langle\nabla_{\eta''},\,\nabla_{\eta''}\rangle,
\]
so that $L_j^v = A_j B$.
By repeated application of the product rule,
\[
(L_j^v)^N = (A_jB)^N
= \sum_{k,\ell} c_{k,\ell}^{(N)}
\Big(\langle 2^{-j/2}\nabla_{\eta'},\,2^{-j/2}\nabla_{\eta'}\rangle\Big)^k
\Big(\langle\nabla_{\eta''},\,\nabla_{\eta''}\rangle\Big)^\ell,
\]
with $k+\ell \le N$.  
Each term has order $2k$ derivatives in $\eta'$ and $2\ell$ derivatives in $\eta''$,
together with a prefactor $2^{-jk}$ coming from the $2^{-j}$ in $A_j$.
Hence each term is of the form
\[
C_{k,\ell}\,2^{-jk}\,\partial_\eta^\alpha,
\qquad |\alpha|=2(k+\ell)\le 2N.
\]
Applying $2^{-jk}\partial_\eta^\alpha$ to the product
\[
\sigma_{\tau j}^v
= \chi_j^v(u(y,\eta))\,\sigma_{\tau j}(\eta,y)\,e^{2\pi i r(\eta)}
\]
and using the Leibniz rule yields a sum of terms
\[
2^{-jk}\,
\partial_\eta^{\alpha_1}\chi_j^v(u(y,\eta))\,
\partial_\eta^{\alpha_2}\sigma_{\tau j}(\eta,y)\,
\partial_\eta^{\alpha_3} e^{2\pi i r(\eta)},
\qquad
\alpha_1+\alpha_2+\alpha_3=\alpha.
\]
We now estimate each factor.\\
\smallskip\noindent
(i) \emph{Symbol term.}  
Since $\sigma_{\tau j}\in S^m_{1/2,1/2}$ and $\langle\eta\rangle\sim 2^j$,
\[
|\partial_\eta^{\alpha_2}\sigma_{\tau j}|\lesssim 
C_{\alpha_2}\,2^{jm}\,2^{-j|\alpha_2|/2}.
\]

\smallskip\noindent
(ii) \emph{Cutoff term.}  
Each derivative in $\eta$ produces factors $2^{-j/2}$ from $u$ and
$2^{j|\alpha|/2}$ from $\partial_u^\alpha\chi_j^v$.  
These cancel, giving
\[
|\partial_\eta^{\alpha_1}\chi_j^v(u(y,\eta))|\lesssim C_{\alpha_1},
\]
uniformly in $j$.

\smallskip\noindent
(iii) \emph{Exponential term.}  
Repeated differentiation gives
\[
\partial_\eta^{\alpha_3} e^{2\pi i r(\eta)}
= e^{2\pi i r(\eta)}\,P_{\alpha_3}\big(
\{\partial_\eta^\gamma r(\eta)\}_{|\gamma|\le |\alpha_3|}
\big),
\]
where $P_{\alpha_3}$ is a fixed polynomial.
Since $r$ is quadratic and its derivatives satisfy the bounds listed above,
\[
|\partial_\eta^{\alpha_3} e^{2\pi i r(\eta)}|
\lesssim C_{\alpha_3}.
\]
Combining the three bounds we obtain
\[
|2^{-jk}\partial_\eta^\alpha(\sigma_{\tau j}^v)|
\lesssim 
2^{-jk}\,C_{\alpha_1,\alpha_3}\,
\big(2^{jm}2^{-j|\alpha_2|/2}\big)
\lesssim
2^{jm}\,2^{-jk}\,2^{-j|\alpha_2|/2}.
\]
Since $|\alpha_2|\le|\alpha|=2(k+\ell)\le2N$, the exponents satisfy
\[
2^{-jk}\,2^{-j|\alpha_2|/2} \leq 1.
\]
Thus each term in the expansion of the action of the operator $(L_j^v)^N$ into the symbol $\sigma_{\tau j}^v(\eta,y)$ is bounded by
\[
C_N\,2^{jm}.
\]
There are finitely many terms in the operator expansion, depending only on $N$.
Summing them yields
\[
|(L_j^v)^N \sigma_{\tau j}^v(\eta,y)|
\le C_N\,2^{jm}.
\] The proof is complete.
\end{proof}
\begin{remark}
Let
\[
A_j := 1 - \langle 2^{-j/2}\nabla_{\eta'},\,2^{-j/2}\nabla_{\eta'}\rangle,
\qquad
B := 1 - \langle\nabla_{\eta''},\,\nabla_{\eta''}\rangle .
\]
For a test function $u=u(\eta',\eta'')\in C^\infty$, we compute
\begin{align*}
A_j B u
&=
\Big(1 - \langle 2^{-j/2}\nabla_{\eta'},\,2^{-j/2}\nabla_{\eta'}\rangle\Big)
\Big(1 - \langle\nabla_{\eta''},\,\nabla_{\eta''}\rangle\Big)u \\
&=
u
- \langle\nabla_{\eta''},\,\nabla_{\eta''}\rangle u
- \langle 2^{-j/2}\nabla_{\eta'},\,2^{-j/2}\nabla_{\eta'}\rangle u \\
&\quad
+ \langle 2^{-j/2}\nabla_{\eta'},\,2^{-j/2}\nabla_{\eta'}\rangle
\langle\nabla_{\eta''},\,\nabla_{\eta''}\rangle u .
\end{align*}
Similarly,
\begin{align*}
B A_j u
&=
\Big(1 - \langle\nabla_{\eta''},\,\nabla_{\eta''}\rangle\Big)
\Big(1 - \langle 2^{-j/2}\nabla_{\eta'},\,2^{-j/2}\nabla_{\eta'}\rangle\Big)u \\
&=
u
- \langle\nabla_{\eta''},\,\nabla_{\eta''}\rangle u
- \langle 2^{-j/2}\nabla_{\eta'},\,2^{-j/2}\nabla_{\eta'}\rangle u \\
&\quad
+ \langle\nabla_{\eta''},\,\nabla_{\eta''}\rangle
\langle 2^{-j/2}\nabla_{\eta'},\,2^{-j/2}\nabla_{\eta'}\rangle u .
\end{align*}
Since derivatives with respect to $\eta'$ commute with those with respect to
$\eta''$, the last terms in the above expressions coincide. Hence,
\[
A_j B u = B A_j u \quad \text{for all } u\in C^\infty,
\]
and therefore
\[
A_j B = B A_j.
\]
This commutation property allows the operators $A_j$ and $B$ to be freely
interchanged and justifies repeated applications of the Leibniz rule to obtain \[
(L_j^v)^N = (A_jB)^N
= \sum_{k,\ell} c_{k,\ell}^{(N)}
\Big(\langle 2^{-j/2}\nabla_{\eta'},\,2^{-j/2}\nabla_{\eta'}\rangle\Big)^k
\Big(\langle\nabla_{\eta''},\,\nabla_{\eta''}\rangle\Big)^\ell,
\]
with $k+\ell \le N$.
\end{remark}

This gives us
\begin{align*}
    \big|K_j^v(x,y)\big| &\lesssim   C_N 2^{jm} \times H_j^v(x,y) \int\limits_{|\eta|\lesssim 2^{-j/2}} \,d\eta,\,m=-\varkappa/2.    
\end{align*}
Therefore, we obtain the estimate
\begin{align*}
    &\big|K_j^v(x,y)\big|\\
    &\lesssim  C_N2^{jm}  \big( 1 + 2^{-j/2} | (\nabla_1 \Phi_\tau(\eta_j^v,y) - x)' | \big)^{-2N} 
\big( 1 + | (\nabla_1 \Phi_\tau(\eta_j^v,y) - x)'' | \big)^{-2N} \int\limits_{|\eta|\lesssim 2^{-j/2}}\,d\eta.
\end{align*}
Now, we are working with polynomial decay in the form given below 
\[
H^v_j(x,y) := \big( 1 + 2^{-j/2} | (\nabla_1 \Phi_\tau(\eta_j^v,y) - x)' | \big)^{-2N} 
\big( 1 + | (\nabla_1 \Phi_\tau(\eta_j^v,y) - x)'' | \big)^{-2N},
\]
where the coordinates are split according to the rank \(\varkappa\) of the non-degenerate part of the phase:
\begin{itemize}
    \item $(\cdot)'\in \mathbb{R}^\varkappa$ denote the \emph{non-degenerate directions}, 
    \item $(\cdot)''\in \mathbb{R}^{d-\varkappa}$ denote the \emph{degenerate directions},
    \item $\eta_j^v$ is the center of the frequency cone,
    \item $N>d/2$ is an integer to ensure convergence of integrals.
\end{itemize}

\begin{lemma}[Integral of the kernel]
Under the above assumptions, there exists a constant \(C>0\), independent of \(j\) and \(y\), such that
\[
\int_{\mathbb{R}^d} H^v_j(x,y) \, dx \le C \, 2^{j \varkappa/2}.
\]
\end{lemma}

\begin{proof}
We aim to compute the integral
\[
\int_{\mathbb{R}^d} H^v_j(x,y) \, dx
\]
and understand how the dyadic frequency scaling affects its size. 
  we apply change of variables $x\mapsto \nabla_1 \Phi_\tau(\eta_j^v,y) - x$ in both non-degenerate and degenerate directions and define the scaled variables as;
\[
u' := 2^{-j/2} (\nabla_1 \Phi_\tau(\eta_j^v,y) - x)' \in \mathbb{R}^\varkappa, 
\qquad
u'' := (\nabla_1 \Phi_\tau(\eta_j^v,y) - x)'' \in \mathbb{R}^{d-\varkappa}.
\]

The Jacobian of this transformation is
\[
dx = dx' \, dx'' = 2^{j \varkappa/2} \, du' \, du''.
\]
 Rewrite the integral in terms of $u', u''$, the integral becomes
\[
\begin{aligned}
\int_{\mathbb{R}^d} H^v_j(x,y) \, dx 
&= \int_{\mathbb{R}^\varkappa} \int_{\mathbb{R}^{d-\varkappa}} 
(1 + |u'|)^{-2N} (1 + |u''|)^{-2N} \, 2^{j \varkappa/2} \, du' \, du'' \\
&= 2^{j \varkappa/2} \int_{\mathbb{R}^{\varkappa}} (1 + |u'|)^{-2N} \, du' 
      \int_{\mathbb{R}^{d-\varkappa}} (1 + |u''|)^{-2N} \, du''.
\end{aligned}
\]
Since \(N > d/2\), both integrals
\[
\int_{\mathbb{R}^{\varkappa}} (1 + |u'|)^{-2N} \, du', 
\quad 
\int_{\mathbb{R}^{d-\varkappa}} (1 + |u''|)^{-2N} \, du''
\]
converge to finite constants $C_\varkappa$ and $C_{d-\varkappa}$, independent of \(j\) and \(y\).\\
Thus, we obtain
\[
\int_{\mathbb{R}^d} H^v_j(x,y) \, dx \le 2^{j \varkappa/2} \, C_\varkappa \, C_{d-\varkappa} =: C \, 2^{j \varkappa/2},
\]
where \(C = C_\varkappa \cdot C_{d-\varkappa}\) is independent of \(j\) and \(y\).\\
The factor \(2^{j \varkappa/2}\) arises from the scaling in the non-degenerate directions, while the polynomial decay ensures integrability in all directions. This completes the proof.
\end{proof}
\begin{proof}[Proof of Theorem \ref{Main:theorem}]
We now have,
\begin{align*}
    \int\big|K_j^v(x,y)\big|dx &\lesssim   C_N 2^{jm} \times 2^{\frac{j\varkappa}{2}} \int\limits_{\eta\in \mathbb{R}^\varkappa:|\eta|\lesssim 2^{-j/2}} \,d\eta,\,m=-\varkappa/2\\&
    \lesssim  C_N 2^{jm} \times 2^{\frac{j\varkappa}{2}} \times 2^{-\frac{j\varkappa}{2}}\\&
    \lesssim C_N 2^{jm}.    
\end{align*}
In view of the analysis above,
summing over the $v$ caps gives us the estimate
\begin{align*}
\sum\limits_{v=1}^{N(v)}\displaystyle{\int}\big|K_{j}^v(x,y)\big|\,dx\le C_N 2^{jm}2^{j\frac{\varkappa}{2}}\le C_N.
\end{align*}


\subsubsection*{Summation over the estimates}

We now explain how the dyadic estimates obtained above can be summed in order
to deduce boundedness of the full operator
\[
T = \sum_{j\ge1} T_j
\quad \text{on } \mathcal{F}L^1(\mathbb{R}^d).
\]

\paragraph{Step 1: A summation estimate in $\mathcal{F}L^1$.}
Let $\chi\in C_c^\infty(\mathbb{R}^d)$ be supported in the annulus
$\{\eta\in\mathbb{R}^d : B_0^{-1}\le |\eta|\le B_0\}$ for some $B_0>0$.
For $f\in\mathcal{S}(\mathbb{R}^d)$, we have
\[
\widehat{\chi(2^{-j}D)f}(\eta)
=
\chi(2^{-j}\eta)\widehat f(\eta).
\]
Hence, by definition of the Fourier--Lebesgue norm,
\begin{align*}
\sum_{j\ge1} \|\chi(2^{-j}D)f\|_{\mathcal{F}L^1}
&=
\sum_{j\ge1}
\int_{\mathbb{R}^d}
|\chi(2^{-j}\eta)|\,|\widehat f(\eta)|\,d\eta.
\end{align*}
Since the integrand is nonnegative, Tonelli's theorem allows us to interchange
the sum and the integral, yielding
\begin{align*}
\sum_{j\ge1} \|\chi(2^{-j}D)f\|_{\mathcal{F}L^1}
&=
\int_{\mathbb{R}^d}
\sum_{j\ge1}
|\chi(2^{-j}\eta)|\,|\widehat f(\eta)|
\,d\eta \\
&=
\int_{\mathbb{R}^d}
\left(
\sum_{j\ge1} |\chi(2^{-j}\eta)|
\right)
|\widehat f(\eta)|\,d\eta.
\end{align*}
By the compact support of $\chi$, for each fixed $\eta\in\mathbb{R}^d$ the sum
$\sum_{j\ge1} |\chi(2^{-j}\eta)|$ contains only finitely many nonzero terms and
is uniformly bounded. Therefore,
\[
\sum_{j\ge1} |\chi(2^{-j}\eta)| \le C,
\quad \text{for all } \eta\in\mathbb{R}^d.
\]
Consequently,
\[
\sum_{j\ge1} \|\chi(2^{-j}D)f\|_{\mathcal{F}L^1}
\le
C
\int_{\mathbb{R}^d} |\widehat f(\eta)|\,d\eta
\lesssim\|f\|_{\mathcal{F}L^1}.
\]

\medskip

\paragraph{Step 2: Estimate for a single dyadic operator.}
Due to the frequency localization of $T_j$, there exists $\chi\in C_c^\infty$
such that  $\chi(\eta)=1$ for $1/2\le |\eta|\le 2$ and $\chi(\eta)=0$ for $|\eta|\le 1/4$ and $|\eta|\ge 4$ (so that $\chi \psi=\psi$). 
Hence,
\[
T_j f = T_j(\chi(2^{-j}D)f).
\]
Using the $L^1$--boundedness of $T_j$, we obtain
\begin{align*}
\|T_j f\|_{\mathcal{F}L^1}
&=
\|T_j(\chi(2^{-j}D)f)\|_{\mathcal{F}L^1} \\
&=
\int_{\mathbb{R}^d}
\left|
\widehat{T_j(\chi(2^{-j}D)f)}(\xi)
\right|
\,d\xi \\
&\le
C
\int_{\mathbb{R}^d}
\left|
\widehat{\chi(2^{-j}D)f}(\xi)
\right|
\,d\xi \\
&\lesssim
\|\chi(2^{-j}D)f\|_{\mathcal{F}L^1}.
\end{align*}
\paragraph{Step 3: Summation over $j$.}
We now estimate the full operator $T=\sum_{j\ge1}T_j$.
By the triangle inequality in $\mathcal{F}L^1$, we have
\begin{align*}
\|Tf\|_{\mathcal{F}L^1}
&=
\left\|
\sum_{j\ge1} T_j f
\right\|_{\mathcal{F}L^1} \\
&=
\int_{\mathbb{R}^d}
\left|
\sum_{j\ge1} \widehat{T_j f}(\xi)
\right|
\,d\xi \\
&\le
\int_{\mathbb{R}^d}
\sum_{j\ge1}
\left|
\widehat{T_j f}(\xi)
\right|
\,d\xi \\
&=
\sum_{j\ge1}
\int_{\mathbb{R}^d}
\left|
\widehat{T_j f}(\xi)
\right|
\,d\xi \\
&=
\sum_{j\ge1}
\|T_j f\|_{\mathcal{F}L^1}.
\end{align*}
Using the estimate obtained in Step~2 and the summation estimate from Step~1,
we conclude that
\[
\|Tf\|_{\mathcal{F}L^1}
\le
C
\sum_{j\ge1}
\|\chi(2^{-j}D)f\|_{\mathcal{F}L^1}
\lesssim
\|f\|_{\mathcal{F}L^1}.
\]
Therefore, the operator $T$ extends to a bounded operator on
$\mathcal{F}L^1(\mathbb{R}^d)$.

\subsubsection*{The case $p=\infty$}

We begin with a simple but fundamental observation concerning frequency localized
functions.
\begin{remark}
Let $\{f_k\}_{k\ge0}\subset \mathcal{S}(\mathbb{R}^d)$ be a sequence such that
\[
\operatorname{supp}\widehat f_0\subset B_2(0),
\qquad
\operatorname{supp}\widehat f_k\subset
\{\eta\in\mathbb{R}^d:2^{k-1}\le|\eta|\le2^{k+1}\},\quad k\ge1.
\]
If the sequence $\{f_k\}$ is bounded in $\mathcal{F}L^\infty(\mathbb{R}^d)$, then
the series $\sum_{k=0}^\infty f_k$ converges in $\mathcal{F}L^\infty(\mathbb{R}^d)$ and
\[
\Big\|\sum_{k=0}^\infty f_k\Big\|_{\mathcal{F}L^\infty}
\lesssim \sup_{k\ge0}\|f_k\|_{\mathcal{F}L^\infty}.
\]
Indeed, at each frequency $\eta\in\mathbb{R}^d$ there are at most two indices $k$
such that $\widehat f_k(\eta)\neq0$.
\end{remark}
Using a dyadic partition of unity $\{\psi_k\}_{k\ge0}$ in the frequency variable, we write
\[
Tf=\sum_{k\ge0}\psi_k(D)Tf,
\]
and therefore
\begin{equation}\label{inf}
\|Tf\|_{\mathcal{F}L^\infty}
\lesssim
\sup_{k\ge0}\|\psi_k(D)Tf\|_{\mathcal{F}L^\infty}
=
\sup_{k\ge0}\Big\|\sum_{j\ge1}\psi_k(D)T_jf\Big\|_{\mathcal{F}L^\infty}.
\end{equation}
Next, we analyze the operators $\psi_k(D)T_j$.
The symbol $\sigma_j(x,\eta)$ of $T_j$ belongs to a bounded subset of
$S^{-\varkappa/2}_{1,0}$, while $\psi_k(\eta)$ belongs to a bounded subset of $S^0_{1,0}$.
By composition of FIOs with pseudodifferential operators,
for each $j,k$ we may decompose
\[
\psi_k(D)T_j = S_{k,j}+R_{k,j},
\]
where:
\begin{itemize}
\item $S_{k,j}$ is a Fourier integral operator with the same phase $\Phi$ and symbol
$\sigma_{k,j}\in S^{-d/2}_{1,0}$, uniformly bounded with respect to $j,k$;
\item $R_{k,j}$ is a smoothing operator whose symbol $r_{k,j}$ belongs to a bounded
subset of $\mathcal{S}(\mathbb{R}^{2d})$.
\end{itemize}
The support properties of $\sigma_{k,j}$ follow from the symbolic calculus.
If $k=0$, then $\sigma_{k,j}$ is supported where
\[\{ (x,\eta)\in \Omega^{\prime}\times \Gamma :
|\nabla_x\Phi_\tau(x,\eta)|\le2,~
2^{j-1}\le|\eta|\le2^{j+1}\}
\]
while for $k\ge1$ it is supported where
\[ \{(x,\eta)\in \Omega^{\prime}\times \Gamma :
2^{k-1}\le|\nabla_x\Phi_\tau(x,\eta)|\le2^{k+1},
2^{j-1}\le|\eta|\le2^{j+1}\}.
\]
By Euler’s identity and the nondegeneracy condition, we have the equivalence
\[
|\nabla_x\Phi_\tau(x,\eta)|
=|\partial^2_{x,\eta}\Phi_\tau(x,\eta)\,\eta|
\asymp|\eta|,
\qquad \forall(x,\eta)\in\Omega \times\Gamma.
\]
Consequently, there exists $N_0>0$ such that
\[
\sigma_{k,j}\equiv0
\qquad\text{whenever } |j-k|>N_0.
\]
On the other hand, the smoothing remainders satisfy
\[
r_k(x,\eta):=\sum_{j=1}^\infty r_{k,j}(x,\eta)\in\mathcal{S}(\mathbb{R}^{2d}),
\]
with $\{r_k\}_{k\ge0}$ bounded in $\mathcal{S}$.
Hence we may write
\[
\sum_{j=1}^\infty\psi_k(D)T_j
=
\sum_{j\ge 1:|j-k|\le N_0} S_{k,j}+r_k(x,D).
\]
Since the number of indices $j$ with $|j-k|\le N_0$ is finite and independent of $k$, the boundedness Theorem \ref{Main:theorem}
for $p=\infty$, together with the uniform boundedness of
$\psi_k(D)$, $R_{k,j}$, and $r_k(x,D)$ on $\mathcal{F}L^\infty$, yields
\[
\Big\|\sum_{j\ge1}\psi_k(D)T_jf\Big\|_{\mathcal{F}L^\infty}
\lesssim
\|f\|_{\mathcal{F}L^\infty},
\qquad \text{uniformly in } k.
\]
Inserting this bound into \eqref{inf}, we conclude that
\[
\|Tf\|_{\mathcal{F}L^\infty}
\lesssim
\|f\|_{\mathcal{F}L^\infty},
\]
which proves the boundedness of $T$ on $\mathcal{F}L^\infty(\mathbb{R}^d)$.

We have already proved Theorem \ref{Main:theorem} in the endpoint cases $p=1$ and $p=\infty$ for Fourier integral operators of order $m=-\varkappa/2$.  
The boundedness result for intermediate exponents $1<p<\infty$ and operators of
order
\[
m=-\varkappa\Big|\frac12-\frac1p\Big|,
\]
will then follow by complex interpolation together with the classical $L^2$
theory of Fourier integral operators.
For $s\in\mathbb{R}$ and $1\le p\le\infty$, we denote by $\mathcal{F}L^p_s$ the
weighted Fourier–Lebesgue space consisting of all tempered distributions $f$
such that
\[
\|f\|_{\mathcal{F}L^p_s}
:=\|\langle\eta\rangle^{s}\widehat f(\eta)\|_{L^p(\mathbb{R}^d)}=\bigg(\int\langle\eta\rangle^{ps}|\hat{f}|^p\,d\eta\bigg)^{1/p}<\infty,
\]
with the usual modification when $p=\infty$.
It is well known that the operator $\langle D\rangle^s$ defines an isomorphism
\[
\langle D\rangle^s:\mathcal{F}L^p_s\longrightarrow \mathcal{F}L^p.
\]
Consequently, an operator $T$ is bounded
\[
T:\mathcal{F}L^p_s\longrightarrow \mathcal{F}L^p
\quad\text{if and only if}\quad
T\langle D\rangle^{-s}:\mathcal{F}L^p\longrightarrow \mathcal{F}L^p
\ \text{is bounded}.
\]
Observe that if $T$ is a Fourier integral operator with phase $\Phi$ and symbol
$\sigma(x,\eta)$ of order $m$, then the composition $T\langle D\rangle^{-s}$ is
again a Fourier integral operator with the \emph{same phase} $\Phi$ and symbol
$\sigma(x,\eta)\langle\eta\rangle^{-s}$, which has order $m-s$.
Let $1<p<2$ and consider a Fourier integral operator $T$ of order
\[
m=-\varkappa\Big(\frac1p-\frac12\Big).
\]
Assume that the boundedness Theorem \ref{Main:theorem} is known for $p=1$ and $p=2$.
As before, for $s\in\mathbb{R}$ we denote by $\mathcal{F}L^p_s$ the weighted
Fourier–Lebesgue space with norm
\[
\|f\|_{\mathcal{F}L^p_s}
=\|\langle\eta\rangle^s\widehat f(\eta)\|_{L^p}.
\]
The operator $\langle D\rangle^s$ defines an isomorphism
$\mathcal{F}L^p_s \to \mathcal{F}L^p$.
By the endpoint estimates, the operator $T$ satisfies
\[
T:\mathcal{F}L^1_{m+\varkappa/2}\longrightarrow \mathcal{F}L^1
\quad\text{and}\quad
T:L^2_m\longrightarrow L^2.
\]
Let $\theta\in(0,1)$ be chosen such that
\[
\frac{1-\theta}{1}+\frac{\theta}{2}=\frac1p.
\]
Solving for $\theta$, we obtain
\[
\frac{1-\theta}{1}+\frac{\theta}{2}
=1-\frac{\theta}{2}
=\frac1p
\quad\Longrightarrow\quad
\theta=2\Big(1-\frac1p\Big).
\]
We now compute the interpolated Sobolev index.
Interpolating between $\mathcal{F}L^1_{m+\varkappa/2}$ and $L^2_m$ yields the exponent
$ 
(1-\theta)(m+\varkappa/2)+\theta m.
$ 
Substituting the value of $m$ and $\theta$, we compute
\begin{align*}
(1-\theta)(m+\varkappa/2)+\theta m
&= m + (1-\theta)\frac \varkappa2 = -\varkappa\Big(\frac1p-\frac12\Big)
+ \frac \varkappa2\Big(1-2\Big(1-\frac1p\Big)\Big) \\
&= -\frac \varkappa p+\frac \varkappa2
+ \frac \varkappa2\Big(\frac{2}{p}-1\Big) \\
&= -\frac \varkappa p+\frac \varkappa2+\frac \varkappa{p}-\frac \varkappa2 = 0.
\end{align*}
Therefore, the complex interpolation argument yields
\[
T:\mathcal{F}L^p\longrightarrow \mathcal{F}L^p,
\]
which proves Theorem \ref{Main:theorem} for all $1<p<2$.
The case $2<p<\infty$ follows by the same argument, interpolating between
$p=2$ and $p=\infty$. 
Finally, if the order $m$ satisfies a strict inequality \[ m \le -\varkappa \bigg|\dfrac{1}{2}-\dfrac{1}{p}\bigg|,\]
the
result follows from the equality case, since any operator of order $m'<m$ is
also of order $m$.
\end{proof}
\noindent\textbf{Conflict of interests statement - Data Availability Statements.}  The authors state that there is no conflict of interest.  Data sharing does not apply to this article as no datasets were generated or
analysed during the current study. \\

\noindent {\bf Acknowledgement.}   Duv\'an Cardona has been  supported  by the FWO  Odysseus  1  grant  G.0H94.18N:  Analysis  and  Partial Differential Equations and by the Methusalem programme of the Ghent University Special Research Fund (BOF)
(Grant number 01M01021), by the FWO Fellowship
Grant No 1204824N and by the FWO Grant K183725N of the Belgian Research Foundation FWO. Frederick Opoku has been supported by Kwame Nkrumah University of Science and Technology, (KNUST) of Ghana.

\bibliographystyle{amsplain}

\end{document}